\documentclass[11pt]{amsart}

 \usepackage{amsfonts,graphics,amsmath,amsthm,amsfonts,amscd,amssymb,amsmath,latexsym,multicol,
 mathrsfs}
\usepackage{epsfig,url}
\usepackage{flafter}
\usepackage{fancyhdr}
\usepackage{hyperref}
\hypersetup{colorlinks=true, linkcolor=black}
\makeatletter

\def\jobis#1{FF\fi
  \def\predicate{#1}%
  \edef\predicate{\expandafter\strip@prefix\meaning\predicate}%
  \edef\job{\jobname}%
  \ifx\job\predicate
}

\makeatother

\if\jobis{proposal}%
\else
\fi

 \usepackage[matrix, arrow]{xy}

\DeclareMathOperator{\Supp}{Supp}

 \numberwithin{equation}{subsection}
 \numberwithin{footnote}{subsection}

 \newtheorem{cor}[subsection]{Corollary}
 \newtheorem{lem}[subsection]{Lemma}
 \newtheorem{prop}[subsection]{Proposition}
 \newtheorem{thm}[subsection]{Theorem}
 \newtheorem{conj}[subsection]{Conjecture}

{%_%_%_%_% upright style; roman (non-italic) text
    \newtheoremstyle{upright}%
        {8pt plus2pt minus4pt}%
        {8pt plus2pt minus4pt}%
        {\upshape}%
        {}%
        {\bfseries\scshape}%
        {}%
        {1em}%
        {}%
\theoremstyle{upright}

}

 \newcommand{\N}{\mathbb N}
 \newcommand{\PP}{\mathbb P}
 
 \newcommand{\Q}{\mathbb Q}
 \newcommand{\R}{\mathbb R}
 \newcommand{\Z}{\mathbb Z}
 \newcommand{\bir}{\dashrightarrow}
\title{Singularities on toric fibrations}
\author{Caucher Birkar \hspace{2cm} Yifei Chen}
\date{\today}
\begin{document}
\maketitle

\emph{Dedicated to the occasion of the 70\textsuperscript{th} birthday of Vyacheslav V. Shokurov}

\begin{abstract}
In this paper we investigate singularities on toric fibrations. In this context we
study a conjecture of Shokurov (a special case of which is due to M\textsuperscript{c}Kernan) which
roughly says that if $(X,B)\to Z$ is an $\epsilon$-lc Fano type log Calabi-Yau fibration, then the singularities
of the log base $(Z,B_Z+M_Z)$ are bounded in terms of $\epsilon$ and $\dim X$ where $B_Z,M_Z$ are the
discriminant and moduli divisors of the canonical bundle formula.
A corollary of our main result says that if $X\to Z$ is a toric Fano fibration with $X$ being $\epsilon$-lc,
then the multiplicities of the fibres over codimension one points are
bounded depending only on $\epsilon$ and $\dim X$.
\end{abstract}

%\tableofcontents

%%%%%%%%%%%%%%%%%%%%%%%%
%%%%%%%%%%%%%%%%%%%%%%%%%%

\section{Introduction}

We work over an algebraically closed field $k$ of characteristic zero.

In this text by a fibration we mean a contraction, that is, a
projective morphism $f\colon X\to Z$ with connected fibres.
A frequently occurring theme in algebraic geometry is to try to understand the geometry of $X$
in terms of $Z$ and the fibres of $f$. This is in particular fundamental in birational geometry where
often one encounters Fano fibrations and Calabi-Yau fibrations which are special instances of
the more general notion of log Calabi-Yau fibrations.
An important problem is to relate the singularities on $X$ to those on the
fibres of $f$ and those on $Z$.

In this context, M\textsuperscript{c}Kernan conjectured that the singularities of
$Z$ are bounded in terms of the singularities of $X$, that is:

\begin{conj}
Let $d$ be a natural number and $\epsilon>0$ be a real number. Then, there is a real number $\delta>0$ depending on $d,\epsilon$ satisfying the following. Assume that
  \begin{itemize}
    \item $X$ is $\epsilon$-lc of dimension $d$ and $\mathbb{Q}$-factorial, and
    \item $f\colon X\rightarrow Z$ is a Mori fibre space, that is, a $K_X$-negative extremal contraction
    with $\dim X>\dim Z$.
  \end{itemize}
  Then $Z$ is $\delta$-lc.
\end{conj}

When $d\leq 2$, the conjecture is trivial as $Z$ is either a point or a smooth curve.
Mori and Prokhorov \cite{MP-conic-bundle} proved the conjecture for $d=3$ and $\epsilon=1$ with $X$ having terminal singularities: in this case one can take $\delta=1$, confirming a conjecture of Iskovskikh.
 Alexeev and Borisov \cite{Alexeev-Borisov-14} proved the
conjecture for toric morphisms of toric varieties.

On the other hand, Shokurov independently proposed a more general conjecture in the setting of pairs.
To state the conjecture we need to recall adjunction for fibrations (also called the canonical bundle formula).
Let $f:X\rightarrow Z$ be a contraction of normal varieties, and $(X,B)$ a klt pair such that $K_X+B\sim_{\mathbb{Q}}0/Z$.
By Kawamata \cite{Kawamata-97}, \cite{Kawamata-98} and Ambro \cite{ambro-adj}\cite{ambro-lc-trivial}, we may write
$$K_X+B\sim_{\mathbb{Q}}f^*(K_Z+B_Z+M_Z)$$
where $B_Z$ is called the \emph{discriminant divisor} and $M_Z$ is called the \emph{moduli divisor}.
The discriminant divisor is determined as a Weil divisor while $M_Z$ is only determined up to
$\R$-linear equivalence. The discriminant divisor is defined by lc thresholds, that is,
the coefficient of a prime divisor $D$ in $B_Z$ is $1-t$ where $t$ is the largest real number such that
$(X,B+tf^*D)$ is lc over the generic point of $D$.
Moreover, for any birational morphism $Z'\to Z$ one can similarly define $B_{Z'},M_{Z'}$ whose pushdown to $Z$
are just $B_Z,M_Z$. Taking $Z'$ to be a sufficiently high resolution,
for any other birational morphism $Z''\to Z'$, $M_{Z''}$ is the pullback of $M_{Z'}$ which are both nef divisors.

Adjunction gives $(Z,B_Z+M_Z)$ which is a generalised pair, see \ref{S:gen-pair}.
To understand the singularities of $(Z,B_Z+M_Z)$ one only needs to understand the discriminant divisors
$B_{Z'}$ on birational models $Z'$. A refined version of Shokurov's conjecture then says:

\begin{conj}
Let $d$ be a natural number and $\epsilon>0$ be a real number. Then
there is a real number $\delta>0$ depending on $d,\epsilon$ satisfying the following.
Let $(X,B)$ be a pair and $f:X\rightarrow Z$ be a contraction such that
  \begin{itemize}
    \item $(X,B)$ is $\epsilon$-lc of dimension $d$,
    \item $K_X+B\sim_{\mathbb{Q}}0/Z$, and
    \item $X$ is of Fano type over $Z$, equivalently, $-K_X$ is big over $Z$.
  \end{itemize}
Then $(Z,B_Z+M_Z)$ is generalised $\delta$-lc.
\end{conj}

The conjecture is equivalent to saying that the discriminant divisors $B_{Z'}$
have coefficients $\le 1-\delta$, for every birational model $Z'\to Z$. In particular, the conjecture
then says that the multiplicities (that is coefficients) of fibres of $f$ over codimension one
points of $Z$ are bounded in terms of $d,\epsilon$. A special case of this is a result of
Mori and Prokhorov \cite{MP-del-pezzo} which says that if $X\to Z$ is a {terminal three-dimensional}
del Pezzo fibration onto a smooth curve, then the multiplicities of the fibres are bounded by 6.

The most general result regarding Shokurov's conjecture is due to the first author. Combining
the main results of \cite{Birkar-14} and \cite{Birkar-BAB} we have:

\begin{thm}
Shokurov conjecture holds when $(F,\Supp B|_F)$ is log bounded where $F$ is a general fibre of $f$.
In particular, the conjecture holds if the coefficients of the horizontal components of $B$ are $\ge \tau$
for some fixed real number $\tau>0$.
\end{thm}

The first assertion follows from \cite{Birkar-14}. The second assertion
 follows from \cite{Birkar-BAB} as in this case $(F,\Supp B|_F)$ is log bounded.\\

In this paper, we prove the following result which essentially says that Shokurov conjecture
holds in the toric setting after taking an average with the toric boundary divisor.
This is enough for some interesting applications.

\begin{thm}\label{T:A-usual}
Let $d$ be a natural number and $\epsilon$ be a positive real number.
Then there exist rational numbers $\alpha,\delta\in (0,1)$ depending only on $d,\epsilon$
satisfying the following. Assume

\begin{itemize}
  \item $(X,\Delta)$ is a projective toric pair of dimension $d$, where $\Delta$ is the torus invariant divisor,
  \item $f\colon X\rightarrow Z$ is a toric contraction,
  \item $(X,B)$ is an $\epsilon$-lc pair where $B$ is a $\mathbb{Q}$-divisor (not necessarily toric), and
  \item $K_X+B\sim_{\mathbb{Q}}0/Z$.
\end{itemize}
  Let
  $$
  \Gamma^{\alpha}=\alpha B+(1-\alpha)\Delta
  $$
  and let
  $$K_X+\Gamma^{\alpha}\sim_{\mathbb{Q}} f^*(K_Z+\Gamma_Z^{\alpha}+N_Z^{\alpha})$$
  be given by adjunction.
  Then $(Z,\Gamma_Z^{\alpha}+N_Z^{\alpha})$ is generalised $\delta$-lc.
\end{thm}

A consequence of the theorem is the following.

\begin{cor}\label{cor-bnd-mult-fib}
  Let $f\colon X\rightarrow Z$ be a toric Fano contraction where $X$ is projective $\epsilon$-lc of dimension $d$.
  Then there exists a natural number $m$ depending only on $\epsilon$ and $d$ such that
  the multiplicities of the fibres of $f$ over codimension one points of $Z$ are bounded by $m$.
\end{cor}

This says that if $D$ is a prime divisor on $Z$, then the coefficients of the components of $f^*D$
mapping onto $D$ are bounded by $m$. In general $f^*D$ may not be defined everywhere as $D$ may not be
$\Q$-Cartier but it is defined over a neighbourhood of the generic point of $D$.

Another corollary of the theorem is a new proof of the following result of
Alexeev and Borisov \cite{Alexeev-Borisov-14}.

\begin{cor}\label{cor-bnd-sing-Z}
  Let $f\colon X\rightarrow Z$ be a toric Fano contraction where $X$ is projective
  $\Q$-factorial $\epsilon$-lc of dimension $d$.
  Then there exists a real number $\delta>0$ depending only on $\epsilon$ and $d$ such that
  $Z$ has $\delta$-lc singularities.
\end{cor}

We will prove a more general form of the above theorem for generalised pairs as follows.

\begin{thm}\label{T:A}
Let $d$ be a natural number and $\epsilon$ be a positive real number.
Then there exist rational numbers $\alpha,\delta\in (0,1)$ depending only on $d,\epsilon$
satisfying the following. Assume

\begin{itemize}
  \item $(X,\Delta)$ is a projective toric pair of dimension $d$, where $\Delta$ is the torus invariant divisor,
  \item $f\colon X\rightarrow Z$ is a toric contraction,
  \item $(X,B+M)$ is a generalised $\epsilon$-lc pair where $B,M'$ are $\mathbb{Q}$-divisors (not necessarily toric)
  and $M'$ is the nef part of the pair, and
  \item $K_X+B+M\sim_{\mathbb{Q}}0/Z$.
\end{itemize}
  Let
  $$
  \Gamma^{\alpha}=\alpha B+(1-\alpha)\Delta \ \ \mbox{and} \ \ N^{\alpha}=\alpha M
  $$
  and let
  $$K_X+\Gamma^{\alpha}+N^{\alpha}\sim_{\mathbb{Q}} f^*(K_Z+\Gamma_Z^{\alpha}+N_Z^{\alpha})$$
  be given by generalised adjunction.
  Then $(Z,\Gamma_Z^{\alpha}+N_Z^{\alpha})$ is generalised $\delta$-lc.
\end{thm}

For the definition of adjunction for generalised pairs, see \ref{S:Adjunction-gen-pair} which follows 
\cite{Filipazzi-18}. 

We stated our results in the projective setting, that is, when $X$ is projective. However,
our proofs also work assuming only projectivity of $X\to Z$. The reason for assuming $X$ projective
is that we apply the adjunction formula for generalised pairs as in \ref{S:Adjunction-gen-pair}
which seems to have been worked out only when the total space is projective but it should hold
in general.

\subsection*{Acknowledgements}
Part of the work was done when we visited Tianyuan Mathematical Center in Northeast China and Jilin University in August 2019. We thank them for their hospitality. Part of the work was done during the second author visiting DPMMS at University of Cambridge in January 2020 and he thanks them for their support.
The first author was supported by a grant of the Royal Society. The second author was supported by NSFC grants (No. 11688101, 11771426 and 11621061). We would like to thank the referee's valuable comments.

\section{Preliminaries}

We work over an algebraically closed field $k$ of characteristic zero: all varieties and schemes are over $k$ unless stated otherwise.

\subsection{Contractions}
By a \emph{contraction} we mean a projective morphism $f\colon X\rightarrow Y$ of varieties such that $f_*\mathcal{O}_X=\mathcal{O}_Y$ ($f$ is not necessarily birational). In particular, $f$ is surjective and has connected fibres.

\subsection{Pairs}
A \emph{sub-pair} $(X,B)$ consists of a normal quasi-projective variety $X$ and a sub-boundary $B$, that is, an $\mathbb{R}$-divisor on $X$ with coefficients in $(-\infty,1]$ such that $K_X+B$ is $\mathbb{R}$-Cartier. We say
$(X,B)$ is a \emph{pair} if in addition $B\ge 0$.

Let $\phi\colon W\to X$ be a log resolution of a sub-pair $(X,B)$. Let $K_W+B_W$ be the
pullback of $K_X+B$. The \emph{log discrepancy} of a prime divisor $D$ on $W$ with respect to $(X,B)$
is defined as
$$
a(D,X,B):=1-\mu_DB_W.
$$
We say $(X,B)$ is \emph{sub-lc} (resp. \emph{sub-klt})(resp. \emph{sub-$\epsilon$-lc}) if
 $a(D,X,B)$ is $\ge 0$ (resp. $>0$)(resp. $\ge \epsilon$) for every $D$. This means that
every coefficient of $B_W$ is $\le 1$ (resp. $<1$)(resp. $\le 1-\epsilon$). If $(X,B)$ is a pair,
we remove the sub and just say it is lc (resp. klt)(resp. $\epsilon$-lc).
Note that since $a(D,X,B)=1$ for most prime divisors, we necessarily have $\epsilon\le 1$.

 We refer to \cite{Kollar-Mori-98} for standard definitions and results on singularities of pairs
 and log minimal model program.

Let $(X,B)$ be an lc sub-pair and $M\geq 0$ be an $\mathbb{R}$-Cartier divisor.
The \emph{lc threshold} of $M$ with respect to $(X,B)$ is the largest real number $t$ so that $(X,B+tM)$ is sub-lc.

\subsection{Fano type varieties}
Let $X\to Z$ be a contraction of normal varieties. We say $X$ is of \emph{Fano type} over $Z$
if there is a boundary $C$ such that $(X,C)$ is a klt pair and $-(K_X+C)$ is ample over $Z$.
It follows from \cite{bchm} that we can run the MMP on any $\R$-Cartier $\R$-divisor $D$ on $X$
relatively over $Z$ and that this terminates with a $D$-negative fibre space or a
$D$-minimal model, that is, a $D$-nef model.

We say $X\rightarrow Z$ is a \emph{Mori fiber space} if $-K_X$ is ample over $Z$, and
the relative Picard number is $\rho(X/Z)=1$.

\subsection{b-divisors}
Let $X$ be a variety. A $b$-$\mathbb{R}$-\emph{Cartier} $b$-\emph{divisor over} $X$ is the choice of a projective birational morphism $Y\rightarrow X$ from a normal variety and an $\mathbb{R}$-Cartier divisor $M$ on $Y$ up to the following equivalence: another projective birational morphism $Y'\rightarrow X$ from a normal variety and an $\mathbb{R}$-Cartier divisor $M'$ defines the same b-$\mathbb{R}$-Cartier b-divisor if there is a common resolution $W\rightarrow Y$ and $W\rightarrow Y'$ on which the pullbacks of $M$ and $M'$ coincide.

\subsection{Generalised pairs} \label{S:gen-pair} The concept of generalised pair is defined and studied by \cite{Birkar-Zhang-16}. For a very recent survey for generalised pairs, we refer to \cite{Birkar-20}.

A \emph{generalised pair} consists of
\begin{itemize}
  \item a normal variety $X$ equipped with a projective morphism $X\rightarrow Z$,
  \item an $\mathbb{R}$-divisor $B\geq 0$ on $X$ with coefficients in $[0,1]$, and
  \item a b-$\mathbb{R}$-Cartier b-divisor over $X$ represented by some projective birational morphism $X'\xrightarrow{\phi} X$ and $\mathbb{R}$-Cartier divisor $M'$ on $X'$
\end{itemize}
 such that $M'$  is nef$/Z$ and $K_X+B+M$ is $\mathbb{R}$-Cartier, where $M:=\phi_*M'$.

We usually refer to the pair by saying $(X,B+M)$ is a generalised pair with data $X'\rightarrow X\rightarrow Z$ and $M'$, and call $M'$ the nef part.

We now define generalised singularities. Replacing $X'$ we can assume $\phi$ is a log resolution of $(X,B)$. We can write
$$K_{X'}+B'+M'=\phi^*(K_X+B+M)$$
for some uniquely determined $B'$. For a prime divisor $D$ on $X'$ the \emph{generalised log discrepancy} $a(D,X,B+M)$ is defined to be $1-\mu_DB'$. We say $(X,B+M)$ is \emph{generalised lc} (resp. \emph{generalised klt}) (resp. \emph{generalised} $\epsilon$-\emph{lc}) if for each $D$ the generalised log discrepancy $a(D,X,B+M)$ is $\geq 0$ (resp. $>0$) (resp. $\geq \epsilon$).

\emph{Generalised sub-pairs} and their singularities are similarly defined by allowing the coefficients of $B$
to be in $(-\infty,1]$.

Let $(X,B+M)$ be a generalised sub-pair. Let $\psi:Y\rightarrow X$ be a generically finite morphism
from a normal variety. Pick a resolution $Y'\to Y$ so that $\rho\colon Y'\bir X'$ is a morphism.
As noted above we can write
$$
K_{X'}+B'+M'=\phi^*(K_X+B+M)
$$
from which we get
$$
K_{Y'}+B_{Y'}+M_{Y'}=\rho^*\phi^*(K_X+B+M)
$$
where $K_{Y'}+B_{Y'}$ is the pullback of $K_{X'}+B'$ and $M_{Y'}$ is the pullback of $M'$.
Let $B_Y$ and $M_Y$ be the pushdowns of $B_{Y'}$ and $M_{Y'}$ to $Y$ respectively.
Then we get
$$
K_Y+B_Y+M_Y=\psi^*(K_X+B+M)
$$
where we consider $(Y,B_Y+M_Y)$ as a generalised sub-pair with nef part $M_{Y'}$.
Using Hurwitz formula it is easy to see that if $(X,B+M)$ is sub-$\epsilon$-lc, then
$(Y,B_Y+M_Y)$ is also sub-$\epsilon$-lc. Indeed assuming $\phi$ is a log resolution,
the coefficients of $B'$ being $\le 1-\epsilon$ implies that the coefficients of $B_{Y'}$
are also $\le 1-\epsilon$.

\subsection{Generalised adjunction for fibrations} \label{S:Adjunction-gen-pair}
 Consider the following set-up. Assume that
\begin{itemize}
  \item $(X,B+M)$ is a generalised sub-pair with data $X'\rightarrow X\to Z$ and $M'$,
  \item $f:X\rightarrow Z$ is a contraction with $\dim Z>0$,
  \item $(X,B+M)$ is generalised sub-lc over the generic point of $Z$, and
  \item $K_X+B+M\sim_{\mathbb{R}}0/Z$.
\end{itemize}
We define the discriminant divisor $B_Z$ for the above setting. Let $D$ be a prime divisor on $Z$.
Let $t$ be the largest real number such that $(X,B+tf^*D+M)$ is generalised lc
over the generic point of $D$. This makes sense even if $D$ is not $\mathbb{Q}$-Cartier because
we only need the pullback $f^*D$ over the generic point of $D$ where $Z$ is smooth.
 We then put the coefficient of $D$ in $B_Z$ to be $1-t$. Note that since $(X,B+M)$ is generalised
 sub-lc over the generic point of $Z$, $t$ is a real number, that is, it is not $-\infty$ or $+\infty$.
 Having defined $B_Z$, we can find $M_Z$ giving
$$K_X+B+M\sim_{\mathbb{R}}f^*(K_Z+B_Z+M_Z)$$
where $M_Z$ is determined up to $\mathbb{R}$-linear equivalence. We call $B_Z$ the \emph{discriminant divisor of adjunction} for $(X,B+M)$ over $Z$. If $B,M'$ are $\Q$-divisors and $K_X+B+M\sim_{\mathbb{Q}}0/Z$, then
$B_Z$ is a $\Q$-divisor and we can choose $M_Z$ also to be a $\Q$-divisor.

For any birational morphism $Z'\to Z$ from a normal variety, we can similarly define $B_{Z'}$ and $M_{Z'}$.

For more details about adjunction for generalised fibrations, we refer to \cite{Filipazzi-18} and
\S 6.1 of \cite{Birkar-18}. It was shown in \cite{Filipazzi-18} that if $X$ is projective,
$(X,B+M)$ is generalised lc over the generic point of $Z$, and $B,M'$ are
$\Q$-divisors, and $M'$ is globally nef, then taking $Z'$ high enough,
$M_{Z'}$ is nef, hence we can regard $(Z,B_Z+M_Z)$ as a generalised pair with nef part $M_{Z'}$.

\begin{lem}\label{l-adj-composition}
Assume that
\begin{itemize}
  \item $(X,B+M)$ is a projective generalised sub-pair with data $X'\rightarrow X$ and $M'$,
  \item $X \xrightarrow{g} Y\xrightarrow{h} Z$ are contractions of normal varieties with $\dim Z>0$,
  \item $(X,B+M)$ is generalised lc over the generic point of $Z$, and
  \item $K_X+B+M\sim_{\mathbb{Q}}0/Z$.
\end{itemize}
Let
$$
K_X+B+M\sim_{\mathbb{Q}} g^*(K_Y+B_Y+M_Y)
$$
be given by adjunction for $(X,B+M)$ over $Y$ and let
$$
K_X+B+M\sim_{\mathbb{Q}}f^*(K_Z+B_Z+M_Z)
$$
be given by adjunction for $(X,B+M)$ over $Z$ where $f=hg$ is $X\to Z$.
Then
$$
K_Y+B_Y+M_Y\sim_{\mathbb{Q}}h^*(K_Z+B_Z+M_Z)
$$
is adjunction for $(Y,B_Y+M_Y)$ over $Z$.

\end{lem}

\begin{proof}
This is essentially Lemma 6.10 of \cite{Birkar-18} where it was stated only for $M=0$.
Note that as noted just before the lemma, the assumptions imply that the moduli divisor
$M_{Y'}$ is nef for a sufficiently high resolution $Y'\to Y$, so it makes sense to consider
adjunction for $(Y,B_Y+M_Y)$ over $Z$.

Let $D$ be a prime divisor on $Z$.
Let $t$ be the largest real number such that $(X,B+{t}f^*D+M)$ is generalised lc over the generic point of $D$.
Similarly, let $s$ be the largest real number such that $(Y,B_Y+{s}h^*D+M_Y)$ is generalised lc over the generic
point of $D$. The lemma says that $s=t$ (and that a similar equality holds for prime divisors on
birational models of $Z$). Since this is a local question around the generic point of $D$,
we can shrink $Z$ around this generic point: we will not need the projectivity of $X,Y,Z$, we only need
the fact that $M',M_{Y'},M_{Z'}$ are nef over $Z$ for high resolutions $X'\to X$, $Y'\to Y$ and $Z'\to Z$.

Shrinking $Z$ as noted, we can assume that $(X',B'+tf^*D+M')$ is generalised sub-lc and that
$B'+tf^*D$ has a component $S$ with coefficient $1$ mapping onto $D$. In particular, this implies that
$(Y,B_Y+th^*D+M_Y)$ is generalised sub-lc but not generalised sub-klt over the generic point of $D$:
indeed we can assume that $X'\to Y$ factors through a high resolution $Y'\to Y$
so that the image of $S$ on $Y$, say $T$, is a prime divisor; but then by definition of
adjunction, the coefficient of $T$ in $B_{Y'}$ is $1$.
Therefore, $t=s$.
\end{proof}

\subsection{Toric varieties and toric morphisms}
A \emph{toric variety} $X$ is given by a pair $(N_X,\Sigma_X)$, where $N_X$ is a lattice and $\Sigma_X$ is a rational polyhedral fan in $N_X\otimes \mathbb{R}$. A \emph{toric morphism} from a toric variety $X$ to a toric variety $Y$ is given by a linear map $F:N_X\rightarrow N_Y$ such that its extension $F_{\mathbb{R}}:N_X\otimes \mathbb{R}\rightarrow N_Y\otimes \mathbb{R}$ sends every cone in the fan $N_X$ to a cone in the fan $\Sigma_Y$. We refer to \cite{fu93}, \cite{Oda-88} or \cite{Cox-Little-Schenck-11} for the general theory of toric varieties.

There is a bijective correspondence between the cones $\sigma$ in $\Sigma_X$ and the orbits $O(\sigma)$ in $X$.
An $m$-dimensional cone $\sigma$ corresponds to a codimension $m$ orbit $O(\sigma)$.
In particular, every ray, that is a one-dimensional cone $\sigma$ determines a torus invariant prime divisor
$\overline{O(\sigma)}$.

If $\Delta$ is the toric boundary divisor on a toric variety $X$, that is, $\Delta$ is the sum of all the
invariant prime divisors, then $(X,\Delta)$ is lc and $K_X+\Delta\sim 0$.

A toric variety $X$ is $\mathbb{Q}$-factorial if and only if the fan $\Sigma_X$ is simplicial, that is,
every cone is a simplex.

A toric Mori fibre space is a Mori fibre space $f:X\rightarrow Y$ where $f$ is a toric
morphism of toric varieties.

It is well-known that toric varieties are Mori dream spaces, that is, if $X\to Z$ is a toric contraction,
then we can run a minimal model program (MMP) on any $\R$-Cartier $\R$-divisor $D$ relatively over $Z$
and that this terminates with
either a $D$-negative fibre space or a $D$-minimal model. This can be derived from the fact that
$X$ is of Fano type over $Z$. Moreover, all the steps of the MMP are toric
because in each step of the MMP the Mori cone is generated by toric extremal rays, cf. Chapter 14 of
\cite{Matsuki} for the $\Q$-factorial case.

\begin{prop}[\cite{Cox-Little-Schenck-11} Proposition 3.3.7] \label{P:toric-isogeny}
  Let $X$ be a toric variety given by $(N_X,\Sigma_X)$.
  Let $N_{X'}$ be a sublattice of finite index in $N_X$ and let $\Sigma_{X'}=\Sigma_X$.
  Let $X'$ be the toric variety determined by $(N_{X'},\Sigma_{X'})$.
  Let $G=N_X/N_{X'}$. Then the toric morphism
  $$\phi:X'=(N_{X'},\Sigma_{X'})\rightarrow X=(N_{X},\Sigma_{X})$$
  induced by the inclusion $N_{X'}\hookrightarrow N_X$ presents $X$ as the quotient $X'/G$.
\end{prop}

Next we consider toric contractions over a torus.

\begin{lem}\label{l-contraction-over-torus}
Let $f\colon X\to Z$ be a toric contraction of toric varieties of dimensions $d,e$ respectively.
If $Z$ is the torus $(k^*)^e$, then $X$ is isomorphic to $F\times Z$ where $F$ is a fibre of $f$.
\end{lem}
\begin{proof}
Assume $X$ is given by $(N_X,\Sigma_X)$ and $Z$ is given by $(N_Z,\Sigma_Z)$ and $f$ given by the
$\Z$-linear map $\pi\colon N_X\to N_Z$. Since $Z$ is the torus, $\Sigma_Z$ consists of only the
zero-dimensional cone. Since $f$ is a contraction, it is surjective, so the image of $\pi$, say $L$, 
is a sublattice of $N_Z$ of finite index. In fact, $L=N_Z$, so $\pi$ is surjective otherwise
 $f$ factors as $(N_X,\Sigma_X)\to (L,\Sigma_Z)\to (N_Z,\Sigma_Z)$
with $(L,\Sigma_Z)\to (N_Z,\Sigma_Z)$ being a finite morphism of degree $>1$ by Proposition \ref{P:toric-isogeny},
a contradiction.

Let $N_F$ be the kernel of $\pi$. Since $Z$ is the torus, every torus invariant subvariety of $X$
maps onto $Z$, that is, each cone in $\Sigma_X$ is mapped to the zero cone in $\Sigma_Z$ which means 
each lattice point in $\Sigma_X$ belongs to $N_F$. 
On the other hand, since $N_Z$ is a {free} abelian group, hence a projective $\Z$-module,
by standard results in commutative algebra, the sequence $0\to N_F\to N_X\to N_Z\to 0$ splits:
indeed taking a basis of $N_F$ together with $e$ elements of $N_X$ mapping onto the basis of $N_Z$,
we get a basis of $N_X$. But then $N_X\simeq N_F\times N_Z$ and $\Sigma_X\simeq \Sigma_F\times \Sigma_Z$
where $\Sigma_F=\Sigma_X$ is considered as a fan in $N_F\otimes \R$. Therefore, $X$ is isomorphic to
$F\times Z$ where $F$ is determined by $(N_F,\Sigma_F)$.
\end{proof}

We also need a result regarding adjunction for toric pairs.

\begin{lem} \label{L:toric-adjunction}
Let $f\colon X\rightarrow Z$ a toric contraction of toric varieties.
Let $\Delta$ be the toric boundary divisor on $X$.
Consider the adjunction formula
$$
K_X+\Delta\sim f^*(K_Z+\Delta_Z+M_Z).
$$
where $\Delta_Z$ is the discriminant divisor and $M_Z$ is the moduli divisor.
Then $\Delta_Z$ is the toric boundary divisor on $Z$ and $M_Z\sim 0$.
\end{lem}
\begin{proof}
First note that in the adjunction formula we wrote $\sim$ instead of $\sim_\Q$.
This is because $K_X+\Delta\sim 0$, so we can choose $M_Z$ so that the formula
is given by $\sim$ instead of $\sim_\Q$.
On the other hand,
it is enough to show that $\Delta_Z$ is the toric boundary divisor on $Z$ because then
$$
0\sim K_X+\Delta\sim f^*(K_Z+\Delta_Z+M_Z)\sim f^*M_Z,
$$
and since $f$ is a contraction this implies that $M_Z\sim 0$.

Let $D$ be a prime divisor on $Z$. Assume $D$ is torus invariant. Then since $f$ is surjective,
there is a torus invariant
subvariety of $X$ mapping onto $D$. Thus the lc threshold $t$ of $f^*D$ with respect to
$(X,\Delta)$ over the generic point of $D$ is $0$, hence the coefficient of $D$ in $\Delta_Z$
is $1$. Thus the toric boundary divisor of $Z$ is contained in the reduced part of $\Delta_Z$.

Now assume $D$ is not an invariant divisor. Shrinking $Z$ we can assume that $Z$ is the torus of dimension
$\dim Z$. Thus $X$ is isomorphic to $F\times Z$ where $F$ is a fibre of $f$, by Lemma \ref{l-contraction-over-torus}.
In particular, the lc threshold $t$ of $f^*D$ with respect to
$(X,\Delta)$ over the generic point of $D$ is $1$, hence the coefficient of $D$ in $\Delta_Z$
is $0$.

An alternative argument is to compactify $X,Z$ and extend $f$ so that $X,Z$ are projective.
In this case, if $\Theta_Z$ is the toric boundary on $Z$, then the first two paragraphs
show that $\Delta_Z-\Theta_Z\ge 0$ and that $\Delta_Z-\Theta_Z+M_Z\sim 0$.
This implies that $\Delta_Z-\Theta_Z=0$ and that $M_Z\sim 0$ as $M_Z$ is pseudo-effective.
\end{proof}

%%%%%%%%%%%%%%%%%%%%%%%%%

\section{Proof of main results}

We start with showing that a generalised version of Shokurov conjecture holds in relative dimension one.

\begin{lem} \label{L:C}
For $\epsilon\in \mathbb{R}_{>0}$, there exists $\delta\in \mathbb{R}_{>0}$ depending only on $\epsilon$
satisfying the following. Assume that
\begin{itemize}
\item $(X,B+M)$ { is projective generalised } $\epsilon$-lc with data $X'\xrightarrow{\phi} X$ and $M'$,

\item $X\to Z$ is a contraction with $\dim X-\dim Z=1$,

\item $K_X+B+M\sim_{\mathbb{Q}}0/Z,$ { where } $B,M'$ { are } $\mathbb{Q}$-divisors, and

\item $X$ is of Fano type over $Z$.
\end{itemize}
Then in the adjunction formula
$$
K_X+B+M\sim_{\mathbb{Q}}f^*(K_Z+B_Z+M_Z),
$$
$(Z,B_Z+M_Z)$ is generalised $\delta$-lc.
\end{lem}

\begin{proof}
Let $D$ be a prime divisor on some high resolution $Z'\to Z$. We can assume $X'\to X$ is a high resolution so that
$f'\colon X'\to Z'$ is a morphism. Write $K_{X'}+B'+M'$ for the pullback of $K_X+B+M$.
 We can assume that $(X',\Supp B'\cup \Supp f'^*D)$ is log smooth.
We want to show that the coefficient of $D$ in $B_{Z'}$ is bounded from above
away from $1$. This coefficient can simply be calculated from the coefficients of $B'$ and $f'^*D$.
By definition of adjunction, it is $1-t$ where $t$ is the lc threshold of $f'^*D$
with respect to $(X',B')$ over the generic point of $D$.

We can assume $X$ is $\Q$-factorial. Since $X$ is of Fano type over $Z$, $-K_X$ is big over $Z$.
So we can write
$$
\phi^*(-K_X)\sim_\Q H'+C'/Z
$$
where $H'$ is ample and $C'\ge 0$. Also write $\phi^*K_X=K_{X'}+E'$ where it is easy to see that $E'\le B'$.
For each small rational number $u>0$ we can then find a general
$$
0\le L'\sim_\Q (1-u)M'+uH'
$$
so that letting
$$
\Delta'=uE'+(1-u)B'+uC'+L',
$$
the sub-pair $(X',\Delta')$ is sub-$\frac{\epsilon}{2}$-lc as
$$
\Delta'\le B'+uC'+L'.
$$
Then over $Z$ we have
$$
K_{X'}+\Delta'=K_{X'}+uE'+(1-u)B'+uC'+L'
$$
$$
\sim_\Q K_{X'}+uE'+(1-u)B'+uC'+(1-u)M'+uH'
$$
$$
\sim_\Q K_{X'}+uE'+(1-u)B'+(1-u)M'-u(K_{X'}+E')
$$
$$
=(1-u)(K_{X'}+B'+M')\sim_\Q 0.
$$
In particular, letting $\Delta=\phi_*\Delta'$, we deduce that $K_{X'}+\Delta'$ is the pullback of $K_X+\Delta$
which implies that $(X,\Delta)$ is $\frac{\epsilon}{2}$-lc.

Now if we choose $u$ to be sufficiently small, then we can assume that the lc threshold $s$ of
$f'^*D$ with respect to $(X',\Delta')$ over the generic point of $D$ is sufficiently close to $t$.
Thus the coefficient of $D$ in the discriminant divisor
$\Delta_{Z'}$, given by adjunction for $(X,\Delta)$ over $Z$, is sufficiently close to the coefficient of $D$
in $B_{Z'}$. Now we can apply Corollary 1.7 of \cite{Birkar-14} to deduce that this coefficient is bounded
from above away from $1$.
\end{proof}

\begin{lem} \label{L:relative-dim}
Let $d,r$ be natural numbers.
Assume that Theorem \ref{T:A} holds when

\begin{itemize}
  \item $\dim X\le d-1$, and also when
  \item  $\dim X= d$ and the relative dimension $\dim X-\dim Z=r$.
\end{itemize}
Then the theorem holds in dimension $d$ when $f:X\rightarrow Z$ can be factored as
$X\xrightarrow{h}V \xrightarrow{g}Z$ where $h,g$ are toric contractions and $h$ is of relative dimension $r$.
\end{lem}
\begin{proof}
By assumption, $h$ is of relative dimension $r\ge 1$, hence
$\dim V\le d-1$. Therefore, by assumption,
Theorem \ref{T:A} holds for both $g,h$ in the following sense.
There exist $\lambda,\xi\in \mathbb{Q}_{>0}$ depending only on $d,\epsilon$ such that
letting $\Gamma^{\lambda}=\lambda B+(1-\lambda)\Delta$ and $N^{\lambda}=\lambda M$, in the adjunction formula
  $$
  K_X+\Gamma^{\lambda}+N^{\lambda}\sim_{\mathbb{Q}}h^*(K_V+\Gamma_V^{\lambda}+N_V^{\lambda}),
  $$
$(V,\Gamma_V^{\lambda}+N_V^{\lambda})$ is generalised $\xi$-lc.
On the other hand, there exist $\beta,\delta\in \mathbb{Q}_{>0}$ depending on $d-r,\xi$, and thus on $d,\epsilon$, such that letting
$$
\Omega_V^{\beta}=\beta \Gamma_V^{\lambda}+(1-\beta)\Delta_V, \ \ \ L_V^{\beta}=\beta N_V^\lambda,
$$
  where $\Delta_V$ is the torus invariant divisor of $V$,
  in the adjunction formula
  $$K_V+\Omega_V^{\beta}+L_V^{\beta}\sim_{\mathbb{Q}}g^*(K_Z+\Omega_Z^{\beta}+L_Z^{\beta}),$$
  $(Z,\Omega_Z^{\beta}+L_Z^{\beta})$ is generalised $\delta$-lc.

  Now let $\alpha=\lambda\beta$,
  $$\Gamma^{\alpha}=\alpha B+(1-\alpha)\Delta, \ \ \ N^{\alpha}=\alpha M
  $$
and consider the adjunction formulas
  $$
  K_X+\Gamma^{\alpha}+N^{\alpha}\sim_{\mathbb{Q}}f^*(K_Z+\Gamma_Z^{\alpha}+N_Z^{\alpha})
  $$
  and
    $$
  K_X+\Gamma^{\alpha}+N^{\alpha}\sim_{\mathbb{Q}}h^*(K_V+\Gamma_V^{\alpha}+N_V^{\alpha}).
  $$

We claim that $(Z,\Gamma_Z^{\alpha}+N_Z^{\alpha})$ is generalised $\delta$-lc. By Lemma \ref{l-adj-composition},
$(Z,\Gamma_Z^{\alpha}+N_Z^{\alpha})$ is also the generalised pair given by adjunction for
$(V,\Gamma_V^{\alpha}+N_V^{\alpha})$ over $Z$. It is then enough to show $(V,\Gamma_V^{\alpha}+N_V^{\alpha})$ has ``better'' singularities than $(V,\Omega_V^{\beta}+L_V^{\beta})$ in a sense that we make precise below.

Let $V'\to V$ be any resolution and $D$ be a prime divisor on $V'$. Taking $X'$ high enough we can assume that the nef 
part $M'$ is on $X'$ and that 
the induced map $h':X'\bir V'$ is a morphism. Let $t$ be the generalised lc threshold of $h'^*D$ with respect to
  $(X',\Gamma'^\lambda+N'^\lambda)$ over the generic point of $D$ where 
  $K_{X'}+\Gamma'^\lambda+N'^\lambda$ is the pullback of   $K_{X}+\Gamma^\lambda+N^\lambda$; more precisely, 
  if $K_{X'}+B'+M'$ is the pullback of $K_X+B+M$ and if $K_{X'}+\Delta'$ is the pullback of $K_X+\Delta$, then  
  $$
  \Gamma'^\lambda=\lambda B'+(1-\lambda)\Delta'  ~~~~~~\mbox{and}~~~~~ N'^\lambda=\lambda M'.
  $$
  By definition, $1-t=b$,
  where $b$ is the coefficient of $D$ in $\Gamma_{V'}'^\lambda$.
  Similarly, let $s$ be the lc threshold of $h'^*D$ with respect to $(X',\Delta')$ over the generic point of $D$. Then $1-s=c$, where $c$ is the coefficient of $D$ in  $\Delta_{V'}$ where $(V',\Delta_{V'})$ is determined by adjunction applied to 
  $(X,\Delta)$ over $Z$. Note however that, by Lemma \ref{L:toric-adjunction}, $K_{V'}+\Delta_{V'}$ is just the pullback of 
  $K_V+\Delta_V$ and $\Delta_V$ is the toric boundary divisor.

  Now
  $$
  (X',\beta \Gamma'^\lambda+\beta N'^\lambda+\beta t h'^*D+(1-\beta)\Delta'+(1-\beta)sh'^*D)
  $$
  is generalised lc over the generic point of $D$, with nef part $\beta N'^\lambda$. 
  That is, the generalised lc threshold of
  $h'^*D$ with respect to
  $$
  (X',\beta \Gamma'^\lambda+\beta N'^\lambda+(1-\beta)\Delta')
  $$
  is at least $\beta t+(1-\beta)s$. But then assuming $K_{X'}+\Gamma'^\alpha+N'^\alpha$ is the 
  pullback of $K_{X}+\Gamma^\alpha+N^\alpha$, we have 
  $$
  \beta \Gamma'^\lambda+(1-\beta)\Delta'=\beta\lambda B'+\beta(1-\lambda)\Delta'+(1-\beta)\Delta'
  =\beta\lambda B'+(1-\beta\lambda)\Delta'=\Gamma'^\alpha,
  $$
  and $\beta N'^\lambda=\beta\lambda M'=N'^\alpha$. So the generalised lc threshold of
  $h'^*D$ with respect to
  $$
  (X',\Gamma'^\alpha+N'^\alpha)
  $$
  is at least $\beta t+(1-\beta)s$.
  Therefore, the coefficient of $D$ in $\Gamma'^\alpha_{V'}$ is
  $$
  \begin{array}{l}\leq 1-\beta t-(1-\beta)s=\beta(1-t)+(1-\beta)(1-s)=\beta b+(1-\beta)c\\
  =\text{coefficient of }D \text{ in }\beta \Gamma'^\lambda_{V'}+(1-\beta)\Delta_{V'}.\end{array}
  $$
Thus we have proved that
$$
\Gamma_{V'}^{\alpha}\leq \beta \Gamma_{V'}^\lambda+(1-\beta)\Delta_{V'}=\Omega_{V'}^{\beta}
$$
where $K_{V'}+\Omega_{V'}^\beta+L^\beta_{V'}$ is the pullback of $K_{V}+\Omega_{V}^\beta+L^\beta_{V}$.

Now let $Z'\to Z$ be any high resolution and $S$ be a prime divisor on $Z'$. Take a high log resolution $V'\to V$ so that 
the induced map 
$g'\colon V'\bir Z'$ is a morphism and so that the nef parts of $(V,\Gamma_V^{\alpha}+N_V^{\alpha})$ and $(V,\Omega_V^{\beta}+L_V^{\beta})$ descent to $V'$. Since
$$
\Gamma_{V'}^{\alpha}\leq \beta \Gamma_{V'}^\lambda+(1-\beta)\Delta_{V'}=\Omega_{V'}^{\beta},
$$
 the generalised lc threshold of $g'^*S$ with respect to
$(V',\Gamma^{\alpha}_{V'}+N^{\alpha}_{V'})$ over the generic point of $S'$ is more than or equal to
the generalised lc threshold of $g'^*S$ with respect to $(V',\Omega^{\beta}_{V'}+L^{\beta}_{V'})$ 
over the generic point of $S'$. This implies that $\Gamma_{Z'}^{\alpha}\leq \Omega_{Z'}^{\beta}$. 
That is, the pair $(Z,\Gamma_Z^{\alpha}+N_Z^{\alpha})$ has ``better'' singularities than
$(Z,\Omega_Z^{\beta}+L_Z^{\beta})$. So $(Z,\Gamma_Z^{\alpha}+N_Z^{\alpha})$ is generalised $\delta$-lc as claimed.
\end{proof}

\begin{lem} \label{L:Pn}
Let $r\in \N$ and $\epsilon\in \R_{>0}$. Then there is a natural number $n$ depending only on $r,\epsilon$
satisfying the following.
Let $f\colon X\rightarrow Z$ be a toric Mori fiber space of relative dimension $r$ where $X$ is
$\Q$-factorial and $\epsilon$-lc. Then there is a finite morphism $\pi\colon X'\to X$ such that
\begin{itemize}
\item $X'$ is $\Q$-factorial,

\item degree of $\pi$ is $\le n$, and

\item the general fibre of the induced morphism $X'\to Z$ is isomorphic to $\PP^r$ as toric varieties.
\end{itemize}
\end{lem}

\begin{proof}
Since $f$ is a contraction, over the torus $U$ in $Z$, $X$ is isomorphic to $F\times U$
where $F$ is a general fibre of $f$, by Lemma \ref{l-contraction-over-torus}. In particular,
there is a one-to-one correspondence between the horizontal torus invariant prime divisors on $X$
and the torus invariant prime divisors on ${F}$. Since $\rho(X/Z)=1$, this implies that
$\rho(F)=1$. Moreover, $F$ is $\Q$-factorial as by the proof of Lemma \ref{l-contraction-over-torus},
its fan $\Sigma_F$ is just a sub-fan of the fan $\Sigma_X$ of $X$
which is simplicial. Therefore, the fan $\Sigma_F$ has exactly $r+1$ rays, say
$R_1,\dots, R_{r+1}$. These are also among the rays of the fan $\Sigma_X$.

 On the other hand, since $f$ is a Mori fibre space and $X$ is $\epsilon$-lc,
 $F$ is a toric Fano variety with $\epsilon$-lc singularities.
Such varieties belong to a bounded family \cite{Borisov-Borisov-93}.
Thus there are only finitely many possibilities for the fan $\Sigma_F$
up to the action of $GL_{r}(\mathbb{Z})$.

Let $v_i$ be the primitive element of the ray $R_i$, $i=1,\ldots,r+1$, in the fan $F$. Then there exist integers $q_i\in \mathbb{Z}$, such that $\sum_{i=1}^{r+1} q_iv_i=0$ and gcd$(q_1,\ldots,q_{r+1})=1$. Moreover, we can assume 
all $q_i$ are positive integers, since $\Sigma_F$ is a complete fan and its cones are generated by the subsets of 
$v_1,\dots,v_{r+1}$ of size $\le r$: if some $q_i$ are positive while some negative, then rearranging indices we find $s$ such that 
the cone generated by $v_1,\dots,v_s$ has a non-trivial intersection with the cone generated by $v_{s+1},\dots,v_{r+1}$ 
which is not possible.  Since the fan $\Sigma_{F}\subseteq \mathbb{R}^{r}$ can take only finitely many possibilities up to the action of $GL_{r}(\mathbb{Z})$ and gcd$(q_1,\ldots,q_{r+1})=1$, such $q_i$ belong to a finite set.

Let $e_1,\dots, e_d$ be a basis of $N_X$. By the proof of {Lemma }\ref{l-contraction-over-torus},
we can choose the $e_i$ so that $e_1,\dots, e_r$ is a basis of $N_F$ and $e_{r+1},\dots,e_d$
map onto a basis of $N_Z$. We can in addition assume that there are finitely many possibilities for
the vectors $q_iv_i$ for $i=1,\dots,r+1$: this can be achieved perhaps after a change of the basis $e_1,\dots,e_r$.

Let $N_{X'}$ be the sublattice of $N_X$ generated by $\{q_1v_1,\ldots,q_{r}v_{r},e_{r+1},\ldots,e_{d}\}$.
Then $N_X/N_{X'}$ is finite with order bounded from above, say by $n$, depending only on $r,\epsilon$.
Let $\Sigma_{X'}$ be the fan $\Sigma_X$ but considered as a fan in $N_{X'}\otimes \R$. Letting $X'$ be
the toric variety associated to $(N_{X'},\Sigma_{X'})$, we get a
finite morphism
$$
\pi\colon X'=(N_{X'},\Sigma_{X'})\rightarrow X=(N_X,\Sigma_X)
$$
of degree $|N_X/N_{X'}|\le n$, by Proposition \ref{P:toric-isogeny}.

Consider the induced toric morphism $f'\colon X'\to Z$:
$$\xymatrix{X'\ar[rr]^{\pi}\ar[drr]_{f'}&&X \ar[d]^{f} \\
&&Z}$$
 Its
general fibre $F'$ is the toric variety associated to $(N_{F'},\Sigma_{F'})$ where $N_{F'}$ is the lattice
generated by $\{q_1v_1,\ldots,q_{r}v_{r}\}$ and $\Sigma_{F'}$ is the fan generated by the rays
$R_1,\dots,R_{r+1}$. We claim that $F'$ is isomorphic to $\PP^r$ as toric varieties.
Consider $\mathbb{Z}^r$ with the standard basis $e_1',\ldots,e_r'$. Define a map $\psi:\Z^r\rightarrow N_{F'}$
of lattices by $\psi(e_i')=q_iv_i$, $i=1,\ldots,r$.
Then
$$
\psi(-e_1'-\cdots-e_{r}')=-\sum_{i=1}^{r}q_iv_i=q_{r+1}v_{r+1}.
$$
Thus we get an isomorphism between $F'$ and the toric variety determined by $\Z^r$ and the fan
given by $e_1',\ldots,e_r', -e_1'-\cdots-e_{r}'$. But the latter is just $\PP^r$.

Finally, note that the fan of $X'$ and $X$ are the same so $\Sigma_{X'}$ is simplicial, hence
$X'$ is $\Q$-factorial.
\end{proof}

\begin{lem} \label{L:P^r-case}
Assume that Theorem \ref{T:A} holds in dimension $\le d-1$.
Then Theorem \ref{T:A} holds in dimension $d$ when the general fibre of $f$ is isomorphic to $\PP^r$ as toric varieties.
\end{lem}
\begin{proof}
By Lemma \ref{L:C}, we can assume that the relative dimension $r=\dim X-\dim Z$ is $\ge 2$.
Moreover, taking a toric $\Q$-factorialisation, we can assume $X$ is $\Q$-factorial.
Since $f$ is a toric morphism, $f^{-1}U\cong F\times U$ where $F$ is a general fibre of $f$ and
$U\subseteq Z$ is the torus, by Lemma \ref{l-contraction-over-torus}. Since $F\simeq \mathbb{P}^r$ by assumption,
we then have $f^{-1}U\cong \mathbb{P}^r\times U$ as toric varieties.
Choose a zero-dimensional torus invariant point $P\in F$.
 Taking the closure of $P\times U\subseteq F\times U$ in $X$, we get a torus invariant closed subvariety $\Pi\subseteq X$. Blowing up $F\times U$ at $P\times U$ gives an extremal contraction $Y_U\rightarrow F\times U$ and a $\mathbb{P}^1$-bundle $Y_U\rightarrow \mathbb{P}^{r-1}\times U$, because blowing up $\mathbb{P}^r$ at a point
 gives a $\mathbb{P}^1$-bundle over $\mathbb{P}^{r-1}$.

The closure $E$ of the exceptional divisor of $Y_U\rightarrow F\times U$ is a toric divisor over $X$,
so it determines an extremal toric divisorial contraction $g:Y\rightarrow X$ with the exceptional divisor $E$.
Over $U$, the two contractions $Y\rightarrow X$ and $Y_U\rightarrow F\times U$ are the same.

Write
$$
\left\{\begin{array}{rcl}
  K_Y+\Delta_Y&=&g^*(K_X+\Delta)\\
  K_Y+B_Y+M_Y&=&g^*(K_X+B+M).
\end{array}\right.
$$
Letting
$$
\Gamma^{\theta}=\theta B+(1-\theta)\Delta \ \ \mbox{and} \ \ N^{\theta}=\theta M,
$$
and
$$
\Gamma^{\theta}_Y=\theta B_Y+(1-\theta)\Delta_Y \ \ \mbox{and} \ \ N^{\theta}_Y=\theta M_Y,
$$
we have
$$
 K_Y+\Gamma_Y^{\theta}+N_Y^{\theta}=g^*(K_X+\Gamma^{\theta}+N^{\theta}).
$$
\

The coefficient of $E$ in $B_Y$ is bounded from below by $1-r$ because the codimension of $\Pi$ in $X$
is $r$. Let $\theta=\frac{1}{r}$. Then since the coefficient of $E$ in $\Delta_Y$ is $1$,
we have $\Gamma_Y^{\theta}\geq 0$ because the coefficient of $E$ in $\Gamma_Y^{\theta}$ is at least
$$
\frac{1}{r}(1-r)+1-\frac{1}{r}=0.
$$
In particular, $(Y,\Gamma_{Y}^{\theta}+N_{Y}^{\theta})$ is a generalised pair with nef part $\theta M'$ 
where we can assume $X'\bir Y$ is a morphism.

By construction, $Y\to Z$ is a toric contraction.
Running an MMP on
$$
K_Y+\Gamma_Y^{\theta}+N_Y^{\theta}-tE
$$
 over $Z$, for some small $t>0$, ends with a toric Mori fiber space $h:Y'\rightarrow W$ because 
 $$
 K_Y+\Gamma_Y^{\theta}+N_Y^{\theta}-tE\equiv -tE/Z.
 $$ 
 The MMP induces an MMP on $Y_U$ over $U$.
Let  $Y'_U$ and $W_U$ be the inverse images of $U$. Then $Y_U'\rightarrow W_U$ is a Mori fiber space.
But note that $Y_U$ has two extremal rays over $U$. One corresponds to $Y_U\rightarrow F\times U$ and the other corresponds
to the $\mathbb{P}^1$-bundle structure over $\mathbb{P}^{r-1}\times U$.
Since
$$
K_Y+\Gamma_Y^{\theta}+N_Y^{\theta}-tE
$$
is ample over $X$, and since $Y_U\rightarrow U$ has only two extremal rays,
the MMP on $Y_U$ simply gives the Mori fiber space $Y_U\rightarrow \mathbb{P}^{r-1}\times U$.
So $Y\dashrightarrow Y'$ is an isomorphism over $U$ and $Y'\rightarrow W$ over $U$ is just $Y_U\rightarrow \mathbb{P}^{r-1}\times U$. Therefore, $Y'\rightarrow W$ has relative dimension one with general fibre $\PP^1$.
Note that $f'\colon Y'\rightarrow Z$ factors as $Y'\rightarrow W\rightarrow Z$ where both
$Y'\rightarrow W$ and $W\rightarrow Z$ are toric contractions.

By construction, $(Y,\Gamma_{Y}^{\theta}+N_{Y}^{\theta})$ is generalised $\frac{\epsilon}{r}$-lc, so 
$(Y',\Gamma_{Y'}^{\theta}+N_{Y'}^{\theta})$ is generalised $\frac{\epsilon}{r}$-lc.
Now recalling that Theorem \ref{T:A} holds in the relative dimension one case by Lemma \ref{L:C}, and
then applying Lemma \ref{L:relative-dim} (by taking $r=1$ in the lemma), we deduce that the
theorem holds for
$$
(Y',\Gamma_{Y'}^{\theta}+N_{Y'}^{\theta})\to Z.
$$
More precisely,
there exist rational numbers $\gamma$ and $\delta>0$ depending only on $d,\epsilon$ (as $r\le d-1$) such that if
$$
\Omega_{Y'}^{\gamma}=\gamma \Gamma_{Y'}^{\theta}+(1-\gamma)\Delta_{Y'}, \ \ \ L_{Y'}^{\gamma}=\gamma N_{Y'}^{\theta},
$$
then in the adjunction formula
$$
K_{Y'}+\Omega_{Y'}^{\gamma}+L_{Y'}^{\gamma}\sim_{\mathbb{Q}}f'^*(K_Z+\Omega_Z^{\gamma}+L_Z^{\gamma}),
$$
$(Z,\Omega_Z^{\gamma}+L_Z^{\gamma})$ is generalised $\delta$-lc.

Defining $\Omega_{Y}^{\gamma}, L_{Y}^{\gamma}$ similarly, we see that $\Omega_{Y'}^{\gamma}, L_{Y'}^{\gamma}$
are the pushdowns of $\Omega_{Y}^{\gamma}, L_{Y}^{\gamma}$. Moreover, we can check that
$$
\Omega_{Y}^{\gamma}=\Gamma_{Y}^{\gamma\theta}:=\gamma\theta B_Y+(1-\gamma\theta)\Delta_{Y}
$$ 
and $L_{Y}^{\gamma}=N_{Y}^{\gamma\theta}:=\gamma\theta M_Y$. By the formulas in the
third paragraph above,
$$
 K_Y+\Gamma_Y^{\gamma\theta}+N_Y^{\gamma\theta}=g^*(K_X+\Gamma^{\gamma\theta}+N^{\gamma\theta})
$$
where $\Gamma^{\gamma\theta},N^{\gamma\theta}$ are defined similar to $\Gamma_Y^{\gamma\theta},N_Y^{\gamma\theta}$.
Therefore, we have the adjunction formula
$$
K_X+\Gamma^{\gamma\theta}+N^{\gamma\theta}\sim_{\mathbb{Q}}f^*(K_Z+\Gamma_Z^{\gamma\theta}+N_Z^{\gamma\theta})
$$
where $\Gamma_Z^{\gamma\theta}=\Omega_Z^{\gamma}$ and $N_Z^{\gamma\theta}=L_Z^{\gamma}$, and $(Z,\Gamma_Z^{\gamma\theta}+N_Z^{\gamma\theta})$ is generalised $\delta$-lc. Now let $\alpha=\gamma \theta$.
\end{proof}

\begin{lem} \label{L:B}
Assume that Theorem \ref{T:A} holds in dimension $\le d-1$. Then
Theorem \ref{T:A} holds in dimension $d$ when $f\colon X\rightarrow Z$ is a toric Mori fiber space
and $X$ is $\mathbb{Q}$-factorial.
\end{lem}
\begin{proof}
Let $\pi\colon \tilde{X}\to X$ be the finite morphism given by Lemma \ref{L:Pn}, and $\tilde{f}$ be the induced morphism
$\tilde{X}\to Z$. Since the general fibre of $\tilde{f}$ is isomorphic to $\PP^r$, $\tilde{f}$ is a contraction.
Write
$$
K_{\tilde{X}}+\tilde{\Delta}=\pi^*(K_X+\Delta)
$$
and
$$
K_{\tilde{X}}+\tilde{B}+\tilde{M}=\pi^*(K_X+B+M)
$$
defined as in \ref{S:gen-pair}. Then $(\tilde{X},\tilde{\Delta})$ is toric, that is, $\tilde{\Delta}$
is the toric boundary of $\tilde X$
but $(\tilde{X},\tilde{B}+\tilde{M})$ is only a generalised sub-pair with generalized sub-$\epsilon$-lc singularities.
There might be components of $\tilde{B}$ with negative coefficient but since $\deg \pi$ is bounded
by some fixed natural number $n$, such coefficients
are bounded from below by Hurwitz formula. Moreover, such components are components of $\tilde{\Delta}$ because 
$\tilde{f}$ can be ramified only along torus invariant divisors.
In particular, letting  $\lambda=\frac{1}{n}$ and letting
$$
\Gamma^{\lambda}=\lambda B+(1-\lambda)\Delta, \ \ \ N^{\lambda}=\lambda M,
$$
and
$$
\tilde{\Gamma}^{\lambda}=\lambda \tilde{B}+(1-\lambda)\tilde{\Delta}, \ \ \ \tilde{N}^{\lambda}=\lambda \tilde{M},
$$
we have
\begin{itemize}
\item $\tilde{\Gamma}^{\lambda}\geq 0$,

\item and
$$
K_{\tilde{X}}+\tilde{\Gamma}^{\lambda}+\tilde{N}^{\lambda}
=\lambda (K_{\tilde{X}}+\tilde{B}+\tilde{M})+(1-\lambda)(K_{\tilde{X}}+\tilde{\Delta})
$$
$$
=\pi^*\lambda (K_{{X}}+{B}+{M})+\pi^*(1-\lambda)(K_{{X}}+{\Delta})=\pi^*(K_X+\Gamma^{\lambda}+N^{\lambda})\sim_\Q 0/Z,
$$
and
\item $(\tilde{X},\tilde{\Gamma}^{\lambda}+\tilde{N}^{\lambda})$ is generalised $\lambda\epsilon$-lc.
\end{itemize}
\

Now by Lemma \ref{L:P^r-case}, there exist rational numbers $\beta,\delta \in(0,1)$ depending only on $\epsilon,d$ (as
$\lambda$ depends only on $n$ which in turn depends only on $r,\epsilon$ hence on $d,\epsilon$) such that
letting 
$$
\tilde{\Gamma}^{\lambda\beta}:=\lambda\beta \tilde{B}+(1-\lambda\beta)\tilde{\Delta}
=\beta\tilde{\Gamma}^{\lambda}+(1-\beta)\tilde{\Delta}
$$
and $\tilde{N}^{\lambda\beta}:=\lambda\beta\tilde{M}=\beta\tilde{N}^{\lambda}$, and considering adjunction 
for 
$$
(\tilde{X},\tilde{\Gamma}^{\lambda\beta}+\tilde{N}^{\lambda\beta})\to Z,
$$
the induced generalised 
pair $(Z,\tilde{\Gamma}_Z^{\lambda\beta}+\tilde{N}_Z^{\lambda\beta})$ is generalised $\delta$-lc. 

Similarly, let 
$$
{\Gamma}^{\lambda\beta}:=\lambda\beta {B}+(1-\lambda\beta){\Delta}
=\beta{\Gamma}^{\lambda}+(1-\beta){\Delta}
$$
and ${N}^{\lambda\beta}:=\lambda\beta{M}=\beta{N}^{\lambda}$, and consider adjunction 
for 
$$
({X},{\Gamma}^{\lambda\beta}+{N}^{\lambda\beta})\to Z,
$$
to get the generalised 
pair $(Z,{\Gamma}_Z^{\lambda\beta}+{N}_Z^{\lambda\beta})$. 

Now $\tilde{\Gamma}_Z^{\lambda\beta}=\Gamma_Z^{\lambda\beta}$ and $\tilde{N}_Z^{\lambda\beta}=N_Z^{\lambda\beta}$,
and similar equalities hold on birational models of $Z$:
indeed if $D$ is a prime divisor on $Z$, then for any real number $t$,
$$
(X,\Gamma^{\lambda\beta}+tf^*D+N^{\lambda\beta})
$$
is generalised lc over the generic point of $D$ if and only if
 $$
 (\tilde{X},\tilde{\Gamma}^{\lambda\beta}+t\tilde{f}^*D+\tilde{N}^{\lambda\beta})
 $$
 is generalised lc over the generic point of $D$; this in turn can be seen by comparing singularities on
 high resolutions of $X,\tilde{X}$; moreover, a similar statement holds for divisors on birational models
 of $Z$. So $\tilde{\Gamma}_Z^{\lambda\beta}=\Gamma_Z^{\lambda\beta}$ which in turn implies 
  $\tilde{N}_Z^{\lambda\beta}=N_Z^{\lambda\beta}$, and the same hold on birational models of $Z$.
Therefore, $(Z,\Gamma_Z^{\lambda\beta}+N_Z^{\lambda\beta})$ is generalised $\delta$-lc.
\end{proof}

\begin{proof}[Proof of Theorem \ref{T:A}]
Let $f\colon X\to Z$ be as in the theorem with $d=\dim X$ and  relative dimension $r=\dim X-\dim Z$.
By induction on dimension we can assume that the theorem holds in dimension $\le d-1$
and that the theorem holds in dimension $d$ and relative dimension $\le r-1$.
Taking a $\mathbb{Q}$-factorization of $X$ we can assume it is $\Q$-factorial.
Run an MMP on $K_X$ over $Z$. The MMP terminates with a toric Mori fibre space.
Replacing $X$ we can then assume we have a Mori fibre space structure $X\to V/Z$ and
that $X$ is $\mathbb{Q}$-factorial.

If $V\to Z$ is birational, then we can replace $Z$ with $V$, so we are done by Lemma \ref{L:B}.
Otherwise the relative dimension $\dim X-\dim V<r$, so we can apply Lemma \ref{L:relative-dim}.
\end{proof}

\begin{proof}(of Theorem \ref{T:A-usual})
This is a special case of Theorem \ref{T:A}.
\end{proof}

\begin{proof}(of Corollary \ref{cor-bnd-mult-fib})
Take a $\Q$-divisor $B$ so that $(X,B)$ is $\epsilon$-lc and $K_X+B\sim_\Q 0/Z$. 
By Theorem \ref{T:A-usual}, there exist rational numbers $\alpha,\delta \in (0,1)$ depending only on 
$d,\epsilon$ such that letting 
  $$
  \Gamma^{\alpha}=\alpha B+(1-\alpha)\Delta
  $$
  and letting
  $$K_X+\Gamma^{\alpha}\sim_{\mathbb{Q}} f^*(K_Z+\Gamma_Z^{\alpha}+N_Z^{\alpha})$$
  be given by adjunction, $(Z,\Gamma_Z^{\alpha}+N_Z^{\alpha})$ is generalised $\delta$-lc.
In particular, 
the coefficients of $\Gamma_Z^{\alpha}$ are bounded from above by $1-\delta$. This means that
for each prime divisor $D$ on $Z$, the lc threshold of $f^*D$ with respect to $(X,B)$ over the
generic point of $D$ is bounded from below by $\delta$ which in particular means that
the multiplicities (=coefficients) of $f^*D$ in components mapping onto $D$ are bounded from above.

\end{proof}

\begin{proof}(of Corollary \ref{cor-bnd-sing-Z})
Take $\alpha,\delta$ and $\Gamma^{\alpha}$ as in the proof of Corollary \ref{cor-bnd-mult-fib} giving the adjunction formula
 $$K_X+\Gamma^{\alpha}\sim_{\mathbb{Q}} f^*(K_Z+\Gamma_Z^{\alpha}+N_Z^{\alpha})$$
where $(Z,\Gamma_Z^{\alpha}+N_Z^{\alpha})$  is $\delta$-lc by Theorem \ref{T:A-usual}. Since $X$ is $\Q$-factorial,
$Z$ is also $\Q$-factorial, hence $Z$ has $\delta$-lc singularities.
\end{proof}

%%%%%%%%%%%%%%%%%%%%%%%%%%%%%%%%%%%%%

\vspace{2cm}

\flushleft{DPMMS}, Centre for Mathematical Sciences,\\
Cambridge University,\\
Wilberforce Road,\\
Cambridge, CB3 0WB,\\
UK\\
email: c.birkar@dpmms.cam.ac.uk

\vspace{1cm}

\flushleft{Hua Loo-Keng} Key Laboratory of Mathematics,\\
{Academy}  of Mathematics and Systems Science,\\
         Chinese Academy of Sciences,\\
         No. 55 Zhonguancun East Road,\\
         Haidian District, Beijing, 100190,\\
         China\\
email: yifeichen@amss.ac.cn


\begin{thebibliography}{LONGEST}

\bibitem[AB14]{Alexeev-Borisov-14} V. Alexeev and A. Borisov; \emph{On the log discrepancies in toric Mori contractions}. Proceedings of the American Mathematical Society. Vol. 12, No. 11, November 2014, pp 3687 -- 3694.

%\bibitem[Am99]{am99} F. Ambro; \emph{On minimal log discrepancies}, Mathematical Research Letters 6, 573 -- 580 (1999).

\bibitem[A05]{ambro-lc-trivial}
F. Ambro; \textit{The moduli b-divisor of an lc-trivial fibration.}
 Compos. Math. \textbf{141} (2005), no. 2, 385--403.

\bibitem[A99]{ambro-adj}
 F. Ambro; \textit{The Adjunction Conjecture and its applications.}
arXiv:math/9903060v3.




\bibitem[Bi14]{Birkar-14} C. Birkar; \emph{Singularities on the base of a Fano type fibration}. J. Reine Angew Math.,   DOI: 10.1515/crelle-2014-0033 (2014), 18 pages.

\bibitem[Bi16]{Birkar-BAB}
C. Birkar; {\emph{Singularities of linear systems and boundedness of Fano varieties.}}
arXiv:1609.05543.

\bibitem[Bi18]{Birkar-18} C. Birkar; \emph{Log Calabi-Yau fibrations}, arXiv:1811.10709v2.

\bibitem[Bi20]{Birkar-20} C. Birkar; \emph{Generalised pairs in birational geometry}, https://www.dpmms.cam.ac.uk/\\
\textasciitilde cb496/2020-gen-pairs-2.pdf


\bibitem[BB93]{Borisov-Borisov-93} A. Borisov and L. Borisov; \emph{Singular toric Fano varieties}. Acad. Sci. USSR Sb. Math. 75 (1993), no. 1, 277 -- 283.


\bibitem[BCHM]{bchm} C. Birkar, P. Cascini, C. D. Hacon and J. M\textsuperscript{c}Kernan, \emph{Existence of minimal models for varieties of log general type}, J. Amer. Math. Soc. 23 no. 2 (2010), 405 -- 468.

\bibitem[BZ16]{Birkar-Zhang-16} C. Birkar and D-Q. Zhang; \emph{Effectivity of Iitaka fibrations and pluricanonical systems of polarized pairs}, Pub. Math. IHES., 123 (2016), 283 --331.

\bibitem[CLS11]{Cox-Little-Schenck-11} D. A. Cox, John. B. Little, and H. K. Schenck; \emph{Toric varieties}, volume 124 of \emph{Graduate Studies in Mathematics}. American Mathematical Society, Provindence, RI, 2011.

\bibitem[Fi18]{Filipazzi-18}  S. Filipazzi; \emph{On a generalized canonical bundle formula and generalized adjunction},    arXiv:1807.04847v3.

\bibitem[Fu93]{fu93} W. Fulton; \emph{Introduction to toric varieties}, Annals of Mathematics Studies, 131. The William H. Rover Lectures in Geometry.
Princeton University Press, Princeton, NJ, 1993.

\bibitem[Ka97]{Kawamata-97} Y. Kawamata; \emph{Subadjunction of log canonical divisors for a variety of codimension 2}, in: Birational algebraic geometry. A conference on algebraic geometry in memory of Wei-Liang Chow (1911-1995) (Baltimore 1996), Contemp. Math. 207, American Mathematical Society, Providence (1997), 79--88.

\bibitem[Ka98]{Kawamata-98} Y. Kawamata; \emph{Subadjuntion of log canonical divisors.} II, Amer. J. Math. 120 (1998), 893 -- 899.

%\bibitem[KMM87]{Kawamata-Matsuda-Matsuki-87} Y. Kawamata, K. Matsuda, and K. Matsuki, \emph{Introduction to the minimal model problem}, Algebraic geometry, Sendai, 1985, Adv. Stud. Pure Math., vol. 10, North-Holland, Amsterdam, 1987, pp. 283 -- 360. MR946243(89e:14015).

\bibitem[KM98]{Kollar-Mori-98} J. Koll\'{a}r and S. Mori, \emph{Birational geometry of algebraic varieties}, with the collabration of C. H. Clemens and A. Corti. Translated from the 1998 Japanese original. Cambridge Tracts in Mathematics, vol. 134, Cambridge University Press, Cambridge, 1998.

\bibitem[M02]{Matsuki}
K. Matsuki; \emph{Introduction to the Mori program}. Universitext. Springer-Verlag, New York, 2002. xxiv+478 pp.


\bibitem[MP08]{MP-conic-bundle}
S. Mori and Yu. Prokhorov; \emph{On $\mathbb{Q}$-conic bundles.} Publ. Res. Inst. Math. Sci.
44 (2008), no. 2, 315--369.


\bibitem[MP09]{MP-del-pezzo}
S. Mori and Yu. Prokhorov; \emph{Multiple fibers of del Pezzo fibrations}. Proceedings of the Steklov Institute of Mathematics, 2009, Vol. 264, pp 131--145.


\bibitem[Od88]{Oda-88} Tadao Oda; \emph{Convex bodies and algebraic geometry}, An introduction to the theory of toric varieties. Translated from the Japanese. Ergebnisse der Mathematik und ihrer Grenzgebiete (3) [Results in Mathematics and Related Areas (3)], vol. 15, Springer-Verlag, Berlin, 1988.


%\bibitem[Sh88]{sh88} V. V. Shokurov, \emph{Problems about Fano varieties}, Birational Geometry of Algebraic Varieties, Open problems --- Katata 1988, 30 -- 32.


\end{thebibliography}
\end{document}

\begin{lem}
  With the setting of Theorem \ref{t-main2}, we can assume that $Z=\mathbb{A}^n$, the affine $n$-space.
\end{lem}

\begin{proof} Since $f:X\rightarrow Z$ is a Mori fiber space, the assumption that $X$ is $\mathbb{Q}$-factorial implies that so is $Z$ (\cite{Kawamata-Matsuda-Matsuki-87} Lemma 5-1-5).  Therefore, the $\mathbb{Q}$-factorial affine toric variety $Z$ is a quotient of $\mathbb{A}^n$ by a finite abelian group $G$. Moreover, the morphism $f:X\rightarrow Z$ can be factored through
$$\xymatrix{X\ar[rr]^{f'}\ar[rrd]_{f}&&\mathbb{A}^n\ar[d]\\
&&Z=\mathbb{A}^n/G}$$

Since $X$ is $\epsilon$-lc, $Z$ is $\delta_{d,n}(\epsilon)$-lc by \cite{Alexeev-Borisov-14} Theorem 1.3, where $\delta$ only depends on $d,n$ and $\epsilon$. It implies that the order of the finite group $G$ is bounded.

Hence, the boundedness of the multiplicity of the special fiber of $f:X\rightarrow Z$ can be reduced to the boundedness of the multiplicity of the special fiber of $f':X\rightarrow \mathbb{A}^n$, where $f':X\rightarrow \mathbb{A}^n$ is a toric Mori fiber space can be derived from the following Proposition.
\end{proof}